\documentclass[reqno,10pt]{amsart}
\usepackage{mathrsfs,amssymb,amsmath,url, enumerate}
\usepackage{verbatim}
\usepackage{color}
\usepackage{hyperref}
\usepackage{tikz}
\usepackage{float}
\usepackage{diagbox}
\usepackage{siunitx}

\newtheorem{thm}{Theorem}

\newtheorem{theorem}{Theorem}[section]
\newtheorem{proposition}[theorem]{Proposition}

\newtheorem{conjecture}[theorem]{Conjecture}
\newtheorem{question}[theorem]{Question}

\theoremstyle{definition}
\newtheorem{example}[theorem]{Example}

\theoremstyle{remark}

\newcommand*{\fplus}{\genfrac{}{}{0pt}{}{}{+}}
\newcommand*{\fdots}{\genfrac{}{}{0pt}{}{}{\cdots}}

\newmuskip\pFqskip
\mathchardef\pFcomma=\mathcode`, 

\hypersetup{
    colorlinks=true, 
    linktoc=all,     
    linkcolor=blue,  
    citecolor = blue
}
\allowdisplaybreaks

\numberwithin{equation}{section}

\begin{document}
\title[The method of telescoping continued fractions]{The method of telescoping continued fractions}

\author[G.~Bhatnagar]{Gaurav Bhatnagar
}
\address{RamanujanExplained.org, 18 Chitra Vihar, Delhi 110092, India.}
\email{bhatnagarg@gmail.com}

\author[K.~Rajkumar]{Krishnan Rajkumar}
\address{Jawaharlal Nehru University,
Delhi, India.}
\email{krishnan.rjkmr@gmail.com}

\begin{abstract}

We give an approach to discover continued fractions for  series of the form
$$\sum_{k=0}^\infty \frac{\epsilon^k}{(x+k)^s},$$
where $\epsilon = \pm 1$. 
We find a continued fraction of the form
\begin{equation*}
	\frac{a_1}{b_1(x)} \fplus \frac{a_2}{b_2(x)} \fplus \fdots
\end{equation*}
where $a_k$ are constants and $b_k(x)$ are polynomials. Our technique involves telescoping continued fractions.

This provides a discovery approach to continued fractions given by Ramanujan for
$$2\sum_{k=1}^\infty \frac{(-1)^{k+1}}{x+2k-1}, 
2\sum_{k=0}^\infty \frac{1}{(x+2k+1)^2}, 2\sum_{k=0}^\infty \frac{(-1)^k}{(x+2k+1)^2},
\sum_{k=1}^\infty \frac{1}{(x+k)^3},
$$
and the like.

We display the first few terms of several continued fractions obtained in this manner for larger values of $s$, including $s=5, 7, 9, 11$.
They do not follow as simple a pattern as Ramanujan's continued fractions.

\end{abstract}


\maketitle


\section{Introduction}

There is a well-known Ramanujan story (told by Ranganathan~\cite[p.~81]{Ranganathan1967}), in which Ramanujan dictates a continued fraction as a solution to a puzzle Mahalanobis reads out to him. Mahalanobis questions Ramanujan:
\begin{quote}
{\sc Mahalanobis}: Did you get the solution in a flash? \\
{\sc Ramanujan}: Immediately I heard the problem, it was clear that the solution was obviously a continued fraction; I then thought, ``Which continued fraction?" and the answer came to my mind. It was just as simple as this.
\end{quote}
In this paper, we give a new technique to answer the question ``Which continued fraction?" corresponds to a given series. 

We apply this technique to give a discovery approach to find continued fractions 
of the form
$$\frac{a_1}{b_1(x)}\fplus\frac{a_2}{b_2(x)}\fplus\frac{a_3}{b_3(x)}\fplus\fdots,$$
where each $b_k(x)$ is a polynomial. The continued fractions
correspond to series of the form
\[
\sum_{k=0}^\infty
\frac{ 1}
{(x+k+1/2)^s}
\quad
\text{and}
\quad
\sum_{k=0}^\infty
\frac{(-1)^k}
{(2x+2k+1)^s},
\]
series related to
\[
\zeta(s)=\sum_{k=1}^{\infty} \frac{1}{k^{s}}
\quad
\text{and}
\quad
\beta(s)=\sum_{k=0}^{\infty} \frac{(-1)^k}{ (2k+1)^{s}}= L(s,\chi_4),
\]
respectively. 
As special cases, we recover (for example) a continued fraction of Ramanujan, given in Chapter 12~\cite[Corollary to Entry 29, p.~149]{Berndt1989} and of Stieltjes~\cite[Corollary to Entry 30, p.~150]{Berndt1989}. 
As another example, we recover a  continued fraction of Ramanujan
\cite[Chapter 12, Entry 32 (iii)]{Berndt1989}, which played an important role in Ap\'{e}ry's proof of the irrationality of $\zeta(3)$ (see~\cite{Rajkumar2012, vanderPoorten1978}). See our remarks in \S\ref{sec:6} for more on Ap\'{e}ry's proof.

We have listed the first few terms of several continued fractions obtained by this method in Appendix~A. Here is one sample. 
\begin{multline*}
\sum_{k=0}^\infty
\frac{ 1}
{(x+k+1/2)^5}
\leftrightsquigarrow
\frac{ \frac{1}{4} }{ x^{4} + \frac{5}{6} x ^{ 2 } -  \frac{47}{144} } \fplus \frac{ \frac{22}{27} }{ x^{2} + \frac{917}{264} } \fplus \frac{ -\frac{32535}{7744} }{ x^{2} + \frac{677361}{106040} } \fplus\\
 \frac{ -\frac{65916928}{4356075} }{ x^{2} + \frac{73206361}{7051660} } \fplus 
\frac{ -\frac{245260880}{6421107} }{ x^{2} + \frac{458251529427}{29777316800} } \fplus \frac{ -\frac{33213134479391559}{414269032960000} }{ x^{2} + \frac{284862092764703087647}{13310627797265915200} } \fplus \fdots .
\end{multline*}
Here \(\leftrightsquigarrow\) means ``corresponds to.''

Our technique develops an idea implicit in Rajkumar~\cite{Rajkumar2012} and explored further by the authors in \cite{BR2023}. It builds on the telescoping method 
given by M\={a}dhava (and his followers in the Kerala school of mathematics) in the context of approximating the M\={a}dhava--Leibniz series $\beta(1)$ for $\pi/4$  (see \cite{krishnachandran2024}). However, the crucial difference is that they looked for generic rational functions instead of the continued fraction in \eqref{eq2}, and ended up with isolated special cases.

We begin with the series
\begin{equation*}
	\sum_{k=0}^{\infty} \frac{\epsilon^{k}}{f(x+k)},
\end{equation*}
 where $\epsilon = \pm 1$ and $f(x)$ is a polynomial of degree $m$, subject to
 \begin{equation}\label{deg-restriction}
\deg f(x)=
 \begin{cases} 
 m\ge 1 & \text{if $\epsilon = -1$},\\
  m\ge 2 & \text{if $\epsilon = 1$}.
  \end{cases}
 \end{equation}
 The restriction on the degree of $f(x)$ ensures the series converges. 

The question is to find a continued fraction for the above series of the following form
\begin{equation}\label{eq1}
	\sum_{k=0}^{\infty} \frac{\epsilon^{k}}{f(x+k)}= 
	\frac{a_1}{b_1(x)} \fplus \frac{a_2}{b_2(x)} \fplus \fdots
\end{equation}
where $a_k$ are constants and $b_k(x)$ are polynomials, and $b_k(x)$ are as `nice' as possible. 

To answer this question, we represent each term of the series as a difference of continued fractions
\begin{equation}\label{eq2}
	 \frac{1}{f(x)}=CF(x)-\epsilon CF(x+1),
\end{equation}
where $CF(x)$ is the continued fraction on the right-hand side of \eqref{eq1}, so that the sum on the left-hand side of \eqref{eq1} becomes a telescoping sum. 
In order to establish \eqref{eq2}, we first calculate the difference \eqref{eq:diff} between the left-hand side and the truncated version of the right-hand side where each continued fraction is replaced by its $k^{\mathrm{th}}$ convergent, and show that the degree of the numerator in \eqref{eq:diff} remains bounded while that of the denominator in \eqref{eq:diff} tends to $\infty$. 

Let
${p_k(x)}/{q_k(x)}$
denote the \(k^{\mathrm{th}}\) convergent of the continued fraction. These are given by:
\begin{subequations}\label{eq:convergent-rec}
\begin{align}
p_k(x)&=b_k(x)p_{k-1}(x)+a_kp_{k-2}(x), \label{eq:convergent-rec-p}\\
q_k(x)&=b_k(x)q_{k-1}(x)+a_kq_{k-2}(x), \label{eq:convergent-rec-q}
\end{align}
\end{subequations}
with initial conditions $p_0(x)=0, q_0(x)=1; p_{-1}(x)= 1, q_{-1}(x)=0$.
(Thus,
$p_1(x)=a_1$, and $q_1(x)=b_1(x).$)

Now, we need to establish that the difference 
\begin{equation}\label{eq:diff}
	\frac{1}{f(x)}- \frac{p_k(x)}{q_k(x)} +\epsilon\frac{ p_k(x+1)}{q_k(x+1)},
\end{equation}
converges to $0$ as $k\rightarrow\infty$. Let
\begin{equation*}
\frac{1}{f(x)}
-\frac{p_k(x)}{q_k(x)}
+\epsilon\frac{p_k(x+1)}{q_k(x+1)}
=
\frac{N_k(x)}
{f(x)q_k(x)q_k(x+1)}.
\end{equation*}

We use the numerator $N_k(x)$ 
to calculate a continued fraction with convergents $p_k(x)/q_k(x)$, for $k=1, 2, 3, \dots$, calculating $(a_k, b_k(x))$ at each step.  

The algorithm chooses \(b_k(x)\) monic and then solves for the constant \(a_k\) and the
remaining coefficients of \(b_k(x)\) by setting the highest possible
coefficients of \(N_k(x)\) equal to zero. The goal is to make the degree of $N_k(x)$ as small as possible. If \(N_k(x)\equiv 0\), then
\eqref{eq2} is exact at the \(k^{\mathrm{th}}\) convergent,
and the algorithm terminates.

The algorithm works, and the results are promising. When  $f(x)=(x+1/2)^m$, and $m$ is small, we recover the nicest classical examples (mentioned above) by this method. The new continued fractions we obtain appear to have some nice properties, too. 
For $k\ge2$, the degree of $b_k(x)$ is $1$ or $2$, and appears to be predictable. Of course, $a_k$ and the coefficients of $b_k$ are rational numbers. The degree of the numerator $N_k(x)$ is bounded, while the degree of the denominator keeps increasing. 
Numerical experiments indicate that the error in \eqref{eq1} between the left-hand side and the $k^{\mathrm{th}}$ convergent of the continued fraction in the right-hand side tends to $0$ as $k\rightarrow\infty$. 
However, the values of $a_k$ and $b_k(x)$ do not appear to have closed forms. This is
only to be expected, 
given that the series $\zeta(s)$ and $\beta(s)$ have drawn a lot of attention. 

We have named our algorithm as Algorithm TCF, where TCF is an acronym for {\em telescoping continued fractions}.

This paper is organized as follows. In \S\ref{sec:2}, we perform some computations to motivate the remainder of the paper. In \S\ref{sec:3}, we list the algorithm. In \S\ref{sec:4}, we prove the algorithm for general $f(x)$, and  in \S\ref{sec:5}, we prove some observations for the interesting case $f(x)=(x+1/2)^m$. We conclude in \S\ref{sec:6} by summarizing our results and discussing some unresolved issues. In Appendix~\ref{appendix}, we list the initial terms of several continued fractions obtained by applying the algorithm.

\section*{Declaration of the use of AI}
We have used ChatGPT-5 Series and GPT 5.6 Family (made by OpenAI) while writing this paper. These AI engines have been used to independently write code to generate examples, to perform symbolic computations, and in generating latex code. The symbolic computations have been checked using Sage and Python. The authors take responsibility for the contents of this paper.

\section{Some explicit computations and examples}\label{sec:2}
In this section, we take $f(x)=(x+1/2)^m$, and present some explicit computations. The purpose is to get a sense of how the algorithm works, and to motivate, by means of examples, what follows in this paper.

Let
\[
CF_1(x)=\frac{a_1}{b_1(x)},
\]
where \(b_1(x)\) is monic of degree \(d\). Then
\begin{equation}\label{N1-gen}
N_1(x)
=
b_1(x)b_1(x+1)
-
a_1\left(x+\frac12\right)^m
\bigl(b_1(x+1)-\epsilon b_1(x)\bigr).
\end{equation}
The first term has degree
\[
\deg\bigl(b_1(x)b_1(x+1)\bigr)=2d.
\]
We now determine $d$.

If \(\epsilon=-1\), then
\(
b_1(x+1)+b_1(x)
\)
has degree \(d\), since \(b_1\) is monic and its leading term is
\(
(x+1)^d+x^d=2x^d+\cdots.
\)
Hence,
\[
\deg\left(
y^m\bigl(b_1(x+1)+b_1(x)\bigr)
\right)
=
m+d.
\]
To obtain a solution, the two highest degrees must be equal, and we must have
\(
d=m.
\)

Similarly, if \(\epsilon=1\), the leading terms of \(b_1(x+1)\) and \(b_1(x)\) cancel, and we have 
\[
\deg\bigl(b_1(x+1)-b_1(x)\bigr)=d-1.
\]
The choice of degree that works is
\(
d=m-1.
\)
Thus we obtain
\begin{equation}\label{eq:deg_b1}
\deg b_1(x)=
\begin{cases}
m, & \epsilon=-1,\\
m-1, & \epsilon=1.
\end{cases}
\end{equation}

\begin{example}
Take \(\epsilon=-1\) and $m=2$, so $\deg b_1(x)=d=2$. Let
\(
b_1(x)=x^2+cx+r.
\)
We obtain:
\begin{multline*}
N_1(x)
=
(1-2a_1)x^4
+
(-2a_1c-4a_1+2c+2)x^3 \\
+
\Big(
-3a_1c-2a_1r-\frac72a_1+c^2+3c+2r+1
\Big)x^2  \\
+
\Big(
-\frac32a_1c-2a_1r-\frac32a_1+c^2+2cr+c+2r
\Big)x
- \\
\frac14a_1c-\frac12a_1r-\frac14a_1+cr+r^2+r.
\end{multline*}
Next we find  \(a_1,b_1(x)\) in such a manner that \(N_1(x)\) is of
the lowest possible degree. By taking the coefficients of \(x^4,x^3,x^2\)
to be \(0\), we obtain the system
\[
\left\{
\begin{aligned}
1-2a_1 &= 0,\\
-2a_1c-4a_1+2c+2 &= 0,\\
-3a_1c-2a_1r-\frac72a_1+c^2+3c+2r+1 &= 0.
\end{aligned}
\right.
\]
This triangular system has the solution
$a_1=\frac12,
c=0,
r=3/4$.
Thus $b_1(x)=x^2+3/4$.
We find that $N_1(x)=1$.

Now we proceed to the second step. We take
\[
CF(x)
=
\frac{1/2}{x^2+\frac34}\fplus \frac{a_2}{b_2(x)}.
\]
Thus
\[
CF(x)
=
\frac{\frac12 b_2(x)}
{\left(x^2+\frac34\right)b_2(x)+a_2}.
\]
We have:
\[
\frac{1}{\left(x+\frac12\right)^2}
=
\frac{\frac12 b_2(x)}
{\left(x^2+\frac34\right)b_2(x)+a_2}
+
\frac{\frac12 b_2(x+1)}
{\left((x+1)^2+\frac34\right)b_2(x+1)+a_2}.
\]
Therefore the numerator is
\[
\begin{aligned}
N_2(x)
={}&
\left(\left(x^2+\frac34\right)b_2(x)+a_2\right)
\left(\left((x+1)^2+\frac34\right)b_2(x+1)+a_2\right)\\
&-
\left(x+\frac12\right)^2
\Bigg[
\frac12 b_2(x)
\left(\left((x+1)^2+\frac34\right)b_2(x+1)+a_2\right)\\
&\hspace{4cm}
+
\frac12 b_2(x+1)
\left(\left(x^2+\frac34\right)b_2(x)+a_2\right)
\Bigg].
\end{aligned}
\]

Again, take \(b_2(x)\) to be a monic polynomial of degree \(d\).
Then the first product has degree
$2d+4$. This matches the degree of the
expression inside the brackets, and the top two terms cancel in the final expression. However, the choice $d=1$ does not lead to a non-zero solution. 
We try $d=2$, and let 
$
b_2(x)=x^2+cx+r
$.
Then we obtain
\[
\begin{aligned}
N_2(x)
={}&
(a_2+1)x^4
+
(a_2c+2a_2+2c+2)x^3\\
&+
\left(
\frac32a_2c+a_2r+\frac{23}{4}a_2+c^2+3c+2r+1
\right)x^2\\
&+
\left(
\frac{15}{4}a_2c+a_2r+\frac{19}{4}a_2+c^2+2cr+c+2r
\right)x\\
&+
a_2^2+\frac{13}{8}a_2c+\frac94a_2r+\frac{13}{8}a_2+cr+r^2+r.
\end{aligned}
\]
Next we find values of \(a_2,c,r\) in such a manner that \(N_2(x)\) is of
the lowest possible degree. By taking the coefficients of \(x^4,x^3,x^2\)
to be \(0\), we obtain the system
\[
\left\{
\begin{aligned}
a_2+1 &= 0,\\
a_2c+2a_2+2c+2 &= 0,\\
\frac32a_2c+a_2r+\frac{23}{4}a_2+c^2+3c+2r+1 &= 0.
\end{aligned}
\right.
\]
The solutions is  $
a_2=-1,
c=0,
r={19}/{4}$.
We obtain:
$b_2(x)=x^2+\frac{19}{4}$, and 
$N_2(x)=16
$.
So in the second step, we obtain
\[
CF(x)
=
\frac{1/2}{x^2+\frac34}\fplus \frac{-1}{x^2+\frac{19}{4}}.
\]
This calculation leads to a continued fraction that is equivalent to \eqref{A2}.
\end{example}

Note three observations in this example.
\begin{enumerate}
\item There was some issue with the choice of $d$ in the Step 2.
\item The equations obtained are triangular at both steps. 
\item There is no linear term in $b_1(x)$ and $b_2(x)$, and these polynomials are even. 
\end{enumerate}

It is instructive to carry out the first step for general $m$.  We will find that the $c=0$ above is no surprise. 
\begin{example}\label{ex:2.2} Take \(\epsilon=-1\) and general $m$. Let \(y=x+{1}/{2}.\)
We explicitly show that the equations obtained from \eqref{N1-gen} have a solution.
Let
\[
b_1(x)=\sum_{j=0}^{m}b_{1,j}x^j,
\qquad
b_{1,m}=1.
\]

We shall construct a solution by restricting \(b_1\) to contain only powers having the same parity as its leading term. Thus we set
\[
b_{1,m-1}=b_{1,m-3}=b_{1,m-5}=\cdots=0.
\]
This is, at present, an ansatz. We shall show shortly that the remaining coefficient equations form a triangular system with non-zero diagonal entries, and hence that a solution of this form indeed exists.

Consider first
\(
b_1(x)b_1(x+1).
\)
We have:
\begin{equation*}
b_1(x)b_1(x+1)
={}
\sum_j b_{1,j}^2x^j(x+1)^j
+
\sum_{j<k}
b_{1,j}b_{1,k}
\left(
x^j(x+1)^k+x^k(x+1)^j
\right).
\end{equation*}

For \(j<k\),
\[
\begin{aligned}
x^j(x+1)^k+x^k(x+1)^j
&=
\bigl(x(x+1)\bigr)^j
\big(
x^{k-j}+(x+1)^{k-j}
\big)\\
&=
\Big(y^2-\frac14\Big)^j
\Big[
\big(y-\frac12\big)^{k-j}
+
\big(y+\frac12\big)^{k-j}
\Big].
\end{aligned}
\]
Since all exponents occurring in \(b_1\) have the same parity, \(k-j\) is even. Thus the expression in square brackets is even in \(y\). The diagonal terms satisfy
\[
x^j(x+1)^j
=
\left(y^2-\frac14\right)^j,
\]
and are also even. Thus
\(
b_1(x)b_1(x+1)
\)
is a polynomial in $y^2$. 

Next, consider
\[
b_1(x)+b_1(x+1)
=
\sum_{s=0}^{\lfloor m/2\rfloor}
b_{1,m-2s}
\left[
\left(y-\frac12\right)^{m-2s}
+
\left(y+\frac12\right)^{m-2s}
\right].
\]
For every \(j\),
\[
\left(y-\frac12\right)^j
+
\left(y+\frac12\right)^j
\]
has the same parity in \(y\) as \(j\). Since each exponent \(m-2s\) has the same parity as \(m\), multiplication by \(y^m\) gives only even powers of \(y\). Thus,
\(
y^m\bigl(b_1(x)+b_1(x+1)\bigr)
\) is also a polynomial in $y^2$. 

Thus, we find that
\(
N_1(y) 
\)
is a polynomial in $y^2$, under the assumption that \(b_1(x)\) contains only the powers that have the same parity as its leading term $m$, which is also the degree of $b_1(x)$. That is \(b_1(x)\) is an odd polynomial if $m$ is odd, and even if $m$ is even. In other words, it is of the same parity as $m$. 

Since $N_1(y)$ is an even polynomial in $y$, all coefficient equations corresponding to odd powers of \(y\) vanish identically, and it remains only to consider the equations obtained from
\[
y^{2m},\quad y^{2m-2},\quad y^{2m-4},\quad\ldots.
\]

Write
\[
b_1\left(y-\frac12\right)
=
\sum_{s=0}^{\lfloor m/2\rfloor}
b_{1,m-2s}
\sum_{p=0}^{m-2s}
\binom{m-2s}{p}
\left(-\frac12\right)^p
y^{m-2s-p},
\]
and
\[
b_1\left(y+\frac12\right)
=
\sum_{t=0}^{\lfloor m/2\rfloor}
b_{1,m-2t}
\sum_{q=0}^{m-2t}
\binom{m-2t}{q}
\left(\frac12\right)^q
y^{m-2t-q}.
\]

The equations are obtained from \eqref{N1-gen} by setting the coefficient of \(y^{2m-2n}\) to $0$. With the above notation, these are given by
\begin{multline}\label{eq:eps-neg1}
\sum_{\substack{s,t\ge0\\s+t\le n}}
b_{1,m-2s}b_{1,m-2t}
\frac{1}{2^{2(n-s-t)}}
\sum_{p=0}^{2(n-s-t)}
(-1)^p
\binom{m-2s}{p}
\binom{m-2t}{2(n-s-t)-p}\\
-
2a_1
\sum_{s=0}^{n}
b_{1,m-2s}
\binom{m-2s}{2(n-s)}
\frac{1}{2^{2(n-s)}} =0,
\end{multline}
where 
\(
n=0,1,\ldots,\left\lfloor\frac m2\right\rfloor
\).

The first equation, corresponding to \(n=0\), is
\(
1-2a_1=0,
\)
which gives
\(
a_1=1/2.
\)

Now let \(n\ge1\). In the \(n^\text{th}\) equation, the new coefficient
\(
b_{1,m-2n}
\)
appears in the quadratic term only when 
\(
(s,t)=(n,0)\) and 
\((s,t)=(0,n),
\)
and so contributes
\(
2b_{1,m-2n}.
\)
In the second term it appears only for \(s=n\), where it contributes
\(
-2a_1b_{1,m-2n}.
\)
Thus, the coefficient of the new unknown \(b_{1,m-2n}\) is
\(
2-2a_1.
\)
Since
\(
a_1=1/2,
\)
this coefficient is equal to \(1\).

It follows that the \(n\)-th equation has the form
\[
b_{1,m-2n}
+
R_n\bigl(
b_{1,m-2},
b_{1,m-4},
\ldots,
b_{1,m-2n+2}
\bigr)
=0,
\]
where \(R_n\) depends only on coefficients determined in the preceding equations. Hence, the system is triangular, with $1$s on the diagonal, and the coefficients
\[
b_{1,m-2},\quad
b_{1,m-4},\quad\ldots
\]
are uniquely determined. 

\end{example}

Thus, we find that the ansatz produces a solution. Again, the equations form a triangular system, with non-zero diagonal entries. Further,
\(
b_1(x)
\)
has the same parity as \(m\), and for this solution
\(
N_1(y)
\)
is an even function of $y$. 
These considerations imply that
\[
\deg N_1(x)\le
\begin{cases}
m-2, & m \text{ even},\\
m-1, & m \text{ odd}.
\end{cases}
\]

\begin{example}\label{ex:2.3} Take \(\epsilon=1\) and general $m$.
This case is very similar. 
We 
begin with the general monic polynomial
\[
b_1(x)=\sum_{j=0}^{d}b_{1,j}x^j,
\qquad
b_{1,d}=1,
\qquad d=m-1.
\]
We can try to construct a solution by restricting \(b_1\) to contain only powers having the same parity as its leading term. Thus, we set
\[
b_{1,d-1}=b_{1,d-3}=b_{1,d-5}=\cdots=0.
\]
As in the case \(\epsilon=-1\), this is, at present, an ansatz. But we can show that the remaining coefficient equations form a triangular system with non-zero diagonal entries, and so a unique solution exists. We show that the parity restriction on $b_1(x)$ makes \(N_1(y)\) an even polynomial in \(y\).

As before 
\(
b_1(x)b_1(x+1)
\)
is even in $y$.
Next, consider
\[
b_1(x+1)-b_1(x)
=
\sum_{s=0}^{\lfloor d/2\rfloor}
b_{1,d-2s}
\left[
\left(y+\frac12\right)^{d-2s}
-
\left(y-\frac12\right)^{d-2s}
\right].
\]
This time, we note that for every \(j\geq1\),
\[
\left(y+\frac12\right)^j
-
\left(y-\frac12\right)^j
\]
has parity opposite to \(j\) as a polynomial in \(y\). Since every exponent
\(
d-2s
\)
has the same parity as \(d=m-1\), the corresponding difference has parity opposite to \(m-1\), and hence, has the same parity as \(m\). Multiplication by \(y^m\) therefore gives only even powers of \(y\). Thus,
\[
y^m\bigl(b_1(x+1)-b_1(x)\bigr) 
\]
is a polynomial in $y^2$, and it follows that
\(
N_1(y) 
\) is even. 

All equations corresponding to odd powers of \(y\), therefore, vanish identically and it remains only to consider the equations obtained from
\[
y^{2d},\quad y^{2d-2},\quad y^{2d-4},\quad\ldots.
\]
In this case, we find that
\[
a_1=\frac1d=\frac1{m-1},
\]
and 
for $n>1$, we find that that the \(n\)-th equation is of the form
\[
\frac{d+2n}{d}\,b_{1,d-2n}
+
R_n\bigl(
b_{1,d-2},
b_{1,d-4},
\ldots,
b_{1,d-2n+2}
\bigr)
=0,
\]
where \(R_n\) depends only on coefficients already determined.  Hence the system is triangular, and 
$b_1(x)$ can be uniquely determined. The details are very similar to the $\epsilon=-1$ case, and are omitted.
\end{example}

Again, we find that the ansatz produces a solution. This time
\(
b_1(x)
\)
has the opposite parity as \(m\), and for this solution
\(
N_1(y)
\)
is an even function of $y$. This implies that
\[
\deg N_1(x)\le
\begin{cases}
m-2, & m \text{ even},\\
m-3, & m \text{ odd}.
\end{cases}
\]

For the next example, we take $m=2$ and continue to the second step. We keep the observations of Example~\ref{ex:2.3} in mind.
\begin{example} 
Take \(\epsilon=1\) and \(f(x)=(x+1/2)^2\), so the degree of the first numerator is given by $d=1$. Let $b_1(x)=x+c$. Since $b_1(x)$ is 
odd, we must have $c=0$. Further, $a_1 = 1/(2-1) = 1$. We find that $N_1(x)=-1/4.$

At the second step, we take
\[
CF(x)
=
\frac{1}{x}\fplus \frac{a_2}{b_2(x)} = \frac{b_2(x)}{xb_2(x)+a_2}.
\]
We have:
\[
\frac{1}{\left(x+\frac12\right)^2}
=
\frac{b_2(x)}{xb_2(x)+a_2}
-
\frac{b_2(x+1)}{(x+1)b_2(x+1)+a_2}.
\]
Therefore the numerator is
\[
\begin{aligned}
N_2(x)
={}&
\bigl(xb_2(x)+a_2\bigr)
\bigl((x+1)b_2(x+1)+a_2\bigr)\\
&-
\left(x+\frac12\right)^2
\Bigl[
b_2(x)\bigl((x+1)b_2(x+1)+a_2\bigr)
-
b_2(x+1)\bigl(xb_2(x)+a_2\bigr)
\Bigr].
\end{aligned}
\]
Again, take \(b_2(x)\) to be a monic polynomial of degree \(d\), so
$
b_2(x)=x^d+\cdots.
$
This time the degrees are $2d+2$ both inside and outside the brackets. The top two terms vanish automatically. The choice of $d=1$ works again. Let $b_2(x)=x+c$. We take $c=0$, since $b_2(x)$ is odd. A short calculation shows that:
$a={1}/{12}, c=0$, so $b_2(x)=x$, $N_2(x)=1/9$. 
So in the second step, we obtain
\[
CF(x)
=
\frac{1}{x}\fplus \frac{1/12}{x}.
\]
Continuing this procedure, we obtain the continued fraction \eqref{A6}. 
\end{example}

Our final examples examine the error in the approximation of the series by the convergents of the desired continued fraction:
\begin{equation}\label{eq:error}
	\eta_k(x):=\sum_{j=0}^{\infty} \frac{\epsilon^{j}}{f(x+j)}-\frac{p_k(x)}{q_k(x)}.
	\end{equation}
It appears that $\eta_k(x) \rightarrow 0$ as $k\rightarrow\infty$. 
Note that choosing a higher value of $x$ roughly corresponds to computing the tail of the relevant series at a later stage. 

We present the numerical values in two examples for which the continued fractions are new.
\begin{example}\label{ex:2.5}
The first example is for \(\epsilon=-1\) and
\(f(x)=(x+\frac12)^3\), which results in a series related to
\(\beta(3)\). The values of the error $\eta_k(x)$ for $x=1.5, 2.5,$ and $10.5$ are reported in Table~\ref{table:1}.

\begin{table}[H]
\centering
\footnotesize
\setlength{\tabcolsep}{1.8pt}
\begin{tabular}{
  |S[table-format=2.1]|
  *{8}{S[table-format=-1.2e-2]|}
}
\hline
{\diagbox{\(x\)}{\(k\)}}
& {1} & {2} & {3} & {4} & {5} & {6} & {7} 
\\ \hline
1.5
& 9.85e-2 & 9.57e-3 & -2.65e-2 & 2.49e-3
& -4.92e-3 & 9.67e-4 & -1.79e-3 
\\ \hline
2.5
& 2.65e-2 & 7.36e-4 & -6.31e-4 & 8.69e-5
& -7.84e-5 & 1.91e-5 & -1.87e-5 
\\ \hline
10.5
& 4.26e-4 & 7.97e-8 & -4.60e-9 & 1.41e-10
& -1.34e-11 & 9.22e-13 & -1.31e-13 
\\ \hline
\end{tabular} 
\caption{Values of \(\eta_k(x)\) for
\(\epsilon=-1\) and \(f(x)=(x+\frac12)^3\).}
\label{table:1}
\end{table}
\end{example}

\begin{example}\label{ex:2.6}
The second example (see Table~\ref{table:2}) concerns \(\epsilon=1\) and
\(f(x)=(x+\frac12)^5\), giving a series related to
\(\zeta(5)\). The values of the error $\eta_k(x)$ are reported for $x=1.5, 2.5,$ and $5.5$. 

\begin{table}[H]
\centering
\footnotesize
\setlength{\tabcolsep}{1.8pt}
\begin{tabular}{
  |S[table-format=1.1]|
  *{8}{S[table-format=-1.2e-2]|}
}
\hline
{\diagbox{\(x\)}{\(k\)}}
& {1} & {2} & {3} & {4} & {5} & {6} & {7} 
\\ \hline
1.5
& 3.69e-2 & -8.87e-4 & -9.02e-5 & -1.79e-5
& -5.00e-6 & -1.74e-6 & -7.05e-7 
\\ \hline
2.5
& 5.68e-3 & -1.12e-5 & -4.18e-7 & -3.57e-8
& -4.93e-9 & -9.36e-10 & -2.24e-10 
\\ \hline
5.5
& 2.66e-4 & -6.86e-9 & -2.36e-11 & -2.46e-13
& -5.30e-15 & -1.94e-16 & -1.06e-17 
\\ \hline
\end{tabular}
\caption{Values of \(\eta_k(x)\) for
\(\epsilon=1\) and \(f(x)=(x+\frac12)^5\).}
\label{table:2}
\end{table}
\end{example}
Although we have not explicitly calculated the error terms, our computations (as illustrated by Examples~\ref{ex:2.5} and \ref{ex:2.6}) indicate that our algorithm generates a continued fraction representation for the given series. 

The examples in this section motivate what follows. The observations on the relationship of the parity of $m$ with the parity of $b_1(x)$, and the fact that $N_1(y)$ is even, generalize for $b_k(x)$ and $N_k(y)$, and are proved in \S\ref{sec:5}. The equations to solve turn out to be triangular; we prove this in \S\ref{sec:4}. We give a better approach to compute the degree of $b_k(x)$, as well as a recursive approach to compute $N_k(x)$ in the algorithm in the following section.

\section{Telescoping continued fractions:  Algorithm TCF}\label{sec:3}
Let \(f(x)\) be a polynomial with degree satisfying \eqref{deg-restriction}, and let \(\epsilon = \pm 1 \). In this section, we describe Algorithm TCF, which constructs a continued fraction satisfying \eqref{eq2}.
Let $p_k$, $q_k$ be given by \eqref{eq:convergent-rec}, and let
 \(N_k(x)\) be the numerator of \eqref{eq:diff}.

The goal is to choose \(a_k\) and \(b_k(x)\) so that \(N_k(x)\) has as small
a degree as possible at every step.

The numerator is defined as:
\begin{subequations}
\begin{equation}\label{eq:Ndef}
	N_k(x) = q_k(x)q_k(x+1) - f(x)\bigg( p_k(x)q_k(x+1)-\epsilon p_k(x+1)q_k(x)\bigg).
\end{equation}
To compute $N_k(x)$, we use two auxiliary polynomials, defined for $k\ge 0$ by:
\begin{equation}\label{eq:def-Ek}
E_k(x)
=
q_k(x)q_{k-1}(x+1)
-
f(x)\bigl(
p_k(x)q_{k-1}(x+1)
-
\epsilon p_{k-1}(x+1)q_k(x)
\bigr);
\end{equation}
\begin{equation}\label{eq:def-Fk}
F_k(x)
=
q_{k-1}(x)q_k(x+1)
-
f(x)\bigl(
p_{k-1}(x)q_k(x+1)
-
\epsilon p_k(x+1)q_{k-1}(x)
\bigr).
\end{equation}
\end{subequations}

The following proposition allows us to compute $N_k(x)$, $E_k(x)$, and $F_k(x)$ recursively. 
\begin{proposition}\label{prop:auxiliary-recurrences}
For \(k\geq 1\), the polynomials \(N_k(x)\), \(E_k(x)\), and \(F_k(x)\) satisfy
\begin{subequations}\label{eq:auxiliary-recurrences}
\begin{align}
N_k(x)
&=
b_k(x)b_k(x+1)N_{k-1}(x)
+
a_k^2N_{k-2}(x) \notag\\
&\qquad
+
a_kb_k(x)E_{k-1}(x)
+
a_kb_k(x+1)F_{k-1}(x),
\label{eq:auxiliary-recurrences-N}\\
E_k(x)
&=
a_kF_{k-1}(x)+b_k(x)N_{k-1}(x),
\label{eq:auxiliary-recurrences-E}\\
F_k(x)
&=
a_kE_{k-1}(x)+b_k(x+1)N_{k-1}(x).
\label{eq:auxiliary-recurrences-F}
\end{align}
\end{subequations}
The initial values are
$N_{-1}(x)=0,$  $N_0(x)=1$, 
$E_0(x)=\epsilon f(x)$, and $F_0(x)=-f(x)$.
\end{proposition}
\begin{proof} The proof is straightforward. We apply the recurrences satisfied by the convergents $p_k(x)$ and $q_k(x)$ given in \eqref{eq:convergent-rec} to $N_k(x)$
(defined in \eqref{eq:Ndef}). This motivates both the definitions \eqref{eq:def-Ek} and  \eqref{eq:def-Fk}, and provides a verification of the recurrences  \eqref{eq:auxiliary-recurrences} as well.
\end{proof}


Now, we need to choose $a_k$ and $b_k(x)$ in the recurrence in such a way that $N_{k}(x)$ has the smallest possible degree. We follow the convention that $b_k(x)$ is monic. 

\subsection*{Algorithm TCF}
For $k\geq 1$, set $$d_k:=\deg(b_k(x))=\begin{cases}
	\deg(E_{k-1}(x))-\deg(N_{k-1}(x)), & \text{if $\epsilon=-1$}; \\
	\deg(E_{k-1}(x))-1-\deg(N_{k-1}(x)) & \text{if $\epsilon=1$}.
\end{cases}$$ 

For this choice of degree $d_k$ of $b_k(x)$, we solve for the top $d_k+1$ coefficients in \eqref{eq:auxiliary-recurrences-N} to be zero to get the $d_k$ non-leading coefficients of $b_k(x)$ and the constant $a_k$. This process continues recursively over $k$, provided there is a unique solution at each stage. If there are no solutions or multiple solutions, the algorithm terminates.

\section{Proof of Algorithm TCF}\label{sec:4}

In this section, we prove that the system of equations at any step in Algorithm TCF, for any polynomial $f(x)$ with rational coefficients of degree $m$ (given by \eqref{deg-restriction}), is triangular, with non-zero diagonal entries and thus, has a unique solution. This is subject to the numerator $N_k(x)$ not being $0$ at any step. This gives a proof of the algorithm.

We retain the notation for the polynomial $f(x)$ of degree $m$ (given by \eqref{deg-restriction}), $\epsilon$, $a_k$, $b_k(x)$, $N_k(x)$, $E_k(x)$, $F_k(x)$ as in \S\ref{sec:3}. In addition, for a fixed $k$, we write
\begin{align*}
E_k(x) &=e_0x^m+e_1x^{m-1}+\cdots,\\
F_k(x) &=f_0x^m+f_1x^{m-1}+\cdots.
\end{align*}
At stage $k\geq1$ of the algorithm, write the monic polynomial $b_k(x)$ as
\[
b_k(x)=x^d+b_1x^{d-1}+b_2x^{d-2}+\cdots+b_d,
\]
where $d=\deg b_k(x)$. 

\begin{theorem}\label{thm:triangular-system}
If $N_{k-1}(x)\neq0$, then the system of equations used to calculate
\[
a_k,b_1,b_2,\ldots,b_d
\]
is triangular, has non-zero diagonal entries, and has a unique solution. Thus, unless $N_k(x)=0$ at some stage, Algorithm TCF never terminates and determines $a_k$ and $b_k(x)$ for every $k\geq1$.

In addition, we have the following:
\begin{itemize}
    \item[(i)]
    \(\displaystyle
    \deg N_k(x)\leq
    \begin{cases}
        m-1,&\epsilon=-1,\\
        m-2,&\epsilon=1;
    \end{cases}
    \)
    \item[(ii)] $\deg E_k(x)=\deg F_k(x)=m$;
    \item[(iii)] $f_0=-\epsilon e_0$, and, if $\epsilon=1$, then
    \(
    e_0(e_1+f_1)\leq0.
    \)
\end{itemize}
\end{theorem}

\begin{proof} The proof is by induction on $k$. The base case $k=0$ is clearly true from the initial conditions $N_0(x)=1$, $E_0(x)=\epsilon f(x)$, and $F_0(x)=-f(x)$.

Now, we assume the truth of the theorem for $0,1,\dots, k-1$ and consider the truth of the theorem for $k$. 
Let
\[
b_k(x)=x^d+b_1x^{d-1}+b_2x^{d-2}+\cdots+b_d.
\]

Then 
\[
b_k(x+1)=x^d
+
\left[
\binom{d}{1}+b_1
\right]x^{d-1}
+
\left[
\binom{d}{2}
+
\binom{d-1}{1}b_1
+
b_2
\right]x^{d-2}
+\cdots
\]

Assume that $N_{k-1}(x)\neq0$. Using the induction hypothesis, we write
\begin{align*}
    N_{k-1}(x)&=n_0x^q+n_1x^{q-1}+\cdots,\\
    E_{k-1}(x)&=e_0x^m+e_1x^{m-1}+\cdots,\\
    F_{k-1}(x)&=-\epsilon e_0x^m+f_1x^{m-1}+\cdots,
\end{align*}
where $q\leq m-1$ if $\epsilon=-1$, $q\leq m-2$ if $\epsilon=1$, and $n_0,e_0\neq0$.

If $k\geq2$, then, provided the algorithm has not terminated at an earlier stage, we may also write
\[
N_{k-2}(x)=p_0x^l+p_1x^{l-1}+\cdots,
\qquad p_0\neq0,
\]
where $l\leq m-1$. For $k=1$, we have
\[
N_{k-2}(x)=N_{-1}(x)=0.
\]

Consider the product
\begin{multline*}
b_k(x)b_k(x+1)
=
\left(x^d+b_1x^{d-1}+b_2x^{d-2}+\cdots
\right) \times \\
\left(
x^d
+
\left[
\binom{d}{1}+b_1
\right]x^{d-1}
+
\left[
\binom{d}{2}
+
(d-1)b_1
+
b_2
\right]x^{d-2}
+\cdots
\right)
 \\
=
x^{2d}+\left(2b_1+\binom{d}{1}\right)x^{2d-1}+\\
\big(2b_2+ c_2(b_1) \big) x^{2d-2}
+\big(2b_3+ c_3(b_1,b_2) \big)x^{2d-3}
+\cdots	
\end{multline*}
where $c_2$ and $c_3$ are explicit algebraic functions.
Hence, the first term $$b_k(x)b_k(x+1)N_{k-1}(x)$$ on the right  in \eqref{eq:auxiliary-recurrences-N} is 
\begin{equation}\label{eq:triangular-first-term}
  n_0 x^{2d+q}+\left(2 n_0 b_1+d_1\right)x^{2d+q-1}+\left(2n_0 b_2+ d_2 \right) x^{2d+q-2}
+\left(2n_0 b_3+ d_3 \right)x^{2d+q-3} + \cdots
	\end{equation}
where each $d_i$ is a function of $d$, the coefficients of $N_{k-1}$ and the previous values $b_1,b_2,\ldots,b_{i-1}$.

Now, we evaluate the last two terms in \eqref{eq:auxiliary-recurrences-N}, depending on the value of $\epsilon$, separately.

\medskip
\noindent
\textbf{Case (i):} $\mathbf{\epsilon=-1}.$

The last two terms $a_k b_k(x) E_{k-1}(x)+a_k b_k(x+1)F_{k-1}(x)$ on the right in \eqref{eq:auxiliary-recurrences-N} give
\begin{equation}\label{eq:triangular-mixed-minus}
	a_k\bigg( 2 e_0 x^{d+m}+ (2 e_0 b_1+ g_1) x^{d+m-1}+ (2 e_0 b_2+ g_2) x^{d+m-2}+ \cdots \bigg)
\end{equation}
where each $g_i$ is a function of $d$, the coefficients of $E_{k-1}, F_{k-1}$ and the previous values $b_1,b_2$, $\ldots,b_{i-1}$.

If $2d+q \neq d+m$, then the larger among these, which is definitely $\geq m$,  would be the degree of $N_k(x)$ (as the remaining term $a_k^2 N_{k-2}(x)$ is $0$ or has degree $l<m$ by the induction hypothesis). Hence, $d$ is chosen such that $2d+q = d+m$, i.e., $d=m-q$ as predicted by the algorithm. 

We then solve the following $d+1$ equations to make the coefficients of $$x^{d+m}, x^{d+m-1},\ldots,x^m$$ to be zero.
\begin{align*}
	n_0+ 2 a_k e_0&=0, \\ b_i(2n_0+2a_k e_0)+h_i&=0, \quad i=1,2,\ldots,d
\end{align*}
where each $h_i$ is a function of $d$, the coefficients of $N_{k-1}, E_{k-1}, F_{k-1}$ and the previous values $a_k,b_1,b_2,\ldots,b_{i-1}$.

Since $N_{k-1}(x)\neq 0$, we have $n_0,e_0 \neq 0$, and so the first equation has a unique solution $a_k\neq 0$. Also, the remaining $d$ equations form a triangular system in $b_1,b_2,\ldots,b_d$ with non-zero diagonal entries given by $2n_0+2a_k e_0 = n_0$. Hence, this system has a unique solution.
Also, the resulting polynomial $N_k(x)$ has degree $\leq m-1$ by this choice of $a_k,b_1, b_2,\dots,b_d$.

From \eqref{eq:auxiliary-recurrences-E} and \eqref{eq:auxiliary-recurrences-F}, we get that
 the coefficient of $x^m$ in both $E_k(x)$ and $F_k(x)$ is $a_k e_0+ n_0=-a_k e_0$, which is non-zero as $a_k,e_0\neq 0$. Hence, 
 $$\deg E_k(x)= \deg F_k(x)=m,$$ and their leading coefficients are equal.
 This completes the induction step in the case $\epsilon=-1$.
 
\medskip
\noindent
\textbf{Case (ii):} $\mathbf{\epsilon=1}$.
  
The last two terms $a_k b_k(x) E_{k-1}(x)+a_k b_k(x+1)F_{k-1}(x)$ on the right in \eqref{eq:auxiliary-recurrences-N} give
\begin{align}\label{eq:triangular-mixed-plus}
	a_k\bigg( &(e_1+f_1-e_0d)x^{d+m-1} \\
	&+\bigl((e_1+f_1-e_0(d-1))b_1+g_1\bigr)x^{d+m-2} \notag\\
	&+\bigl((e_1+f_1-e_0(d-2))b_2+g_2\bigr)x^{d+m-3}
	+\cdots \bigg) \notag
\end{align}
where each $g_i$ is a function of $d$, the coefficients of $E_{k-1}, F_{k-1}$ and the previous values $b_1,b_2,$ $\ldots,b_{i-1}$.

Similar to the previous case, if $2d+q \neq d+m-1$, then the larger of these which is $\geq m-1$ would be the degree of $N_k(x)$ and no cancellation takes place. On the other hand, if $d$ is chosen such that $2d+q = d+m-1$, i.e. $d=m-q-1$ as predicted by the algorithm, we can then solve the following $d+1$ equations to make the coefficients of $x^{d+m-1}, x^{d+m-2},\ldots,x^{m-1}$ to be zero.
\begin{align*}
	n_0+a_k(e_1+f_1-e_0d)&=0,\\
	b_i\bigl(2n_0+a_k(e_1+f_1-e_0(d-i))\bigr)+h_i&=0,
	\quad i=1,2,\ldots,d
\end{align*}
where each $h_i$ is a function of $d$, the coefficients of $N_{k-1}, E_{k-1}, F_{k-1}$ and the previous values $a_k,b_1,b_2,\ldots,b_{i-1}$.


As in the previous case, since $N_{k-1}(x) \neq 0$, we have $n_0 \neq 0$ and $e_1+f_1-e_0 d$ has magnitude $\geq |e_0|d$, as $e_0$ and $e_1+f_1$ have opposite signs or the latter is $0$, from the induction hypothesis. Hence, $e_1+f_1-e_0 d \neq 0$ and we get a unique solution $a_k\neq 0$. Also, the remaining $d$ equations form a triangular system in $b_1,b_2,\ldots,b_d$ with non-zero diagonal entries 
$$2n_0+ a_k(e_1+f_1-e_0(d-i)) = a_k(e_0 (d+i)-e_1-f_1)\neq 0$$ (for the same reason that $e_1+f_1$ has the opposite sign as $e_0$ or is $0$). Hence, this system has a unique solution.
Again, the resulting polynomial $N_k(x)$ has degree $\leq m-2$ by this choice of $a_k,b_1, b_2,\dots,b_d$.

From \eqref{eq:auxiliary-recurrences-E} and \eqref{eq:auxiliary-recurrences-F}, and the choice of $d=m-q-1$, it follows the terms $b_k(x)N_{k-1}(x)$, $b_k(x+1)N_{k-1}(x)$ on the right-hand side both have degree $m-1$. Hence, 
$$\deg E_k(x)= \deg F_k(x)=m,$$ and the leading coefficients of $E_k(x),F_k(x)$ are $-a_k e_0, a_k e_0$ (respectively). 

Finally, we have the coefficient of $x^{m-1}$ in $E_k(x)+F_k(x)$ is 
\begin{equation*}
	a_k(f_1+e_1)+2 n_0 = a_k(2 e_0 d-e_1-f_1),
\end{equation*}
which has the same sign as $a_k e_0$ and hence the opposite sign as the leading coefficient of $E_k(x)$. 

This completes the induction step in this case.
\end{proof}

To summarize, Algorithm TCF does not terminate as long as $N_k(x)\neq 0$ at every step.
Next, we specialize the function $f(x)$  to obtain some more information about the resulting continued fractions.

\section{Properties of the continued fractions when $f(x) = (x+1/2)^m$}\label{sec:5}
We applied Algorithm TCF in the case where $f(x)=(x+1/2)^m$, with $\epsilon=\pm 1$ for several small values of $m$. The first few terms of the results are listed in Appendix~\ref{appendix}. The objective of this section is to establish some patterns that are visible from these examples.

We first generalize the observations in Examples~\ref{ex:2.2} and \ref{ex:2.3}, and consider $b_k(x)$ and $N_k(x)$, for $k\ge 1$. 

	\begin{proposition}\label{prop:1}
	For $f(x)=(x+1/2)^m, \epsilon=-1$, and $k\geq 1$, we have:
	\begin{itemize}
		\item[(a)] $b_k(x)$ has the same parity as $m$.
		\item[(b)] $N_k(x)$ is an even function of $(x+1/2)$.
		\item[(c)] For $j\geq 0$, both $x^j E_k(x)+(x+1)^j F_k(x)$ and $(x+1)^j E_k(x)+x^j F_k(x)$ are even functions of $(x+{1}/{2})$ whenever $j$ has the same parity as $m$, and odd functions of $(x+{1}/{2})$ whenever $j$ has the opposite parity as $m$.
	\end{itemize}
\end{proposition}
	\begin{proposition}\label{prop:2}
	For $f(x)=(x+1/2)^m, \epsilon=1$, and $k\geq 1$, we have:
	\begin{itemize}
		\item[(a)] $b_k(x)$ has the opposite parity as $m$.
		\item[(b)] $N_k(x)$ is an even function of $(x+1/2)$.
		\item[(c)] For $j\geq 0$, both $x^j E_k(x)+(x+1)^j F_k(x)$ and $(x+1)^j E_k(x)+x^j F_k(x)$ are even functions of $(x+{1}/{2})$ whenever $j$ has the opposite parity as $m$, and odd functions of $(x+{1}/{2})$ whenever $j$ has the same parity as $m$.
	\end{itemize}
\end{proposition}

\begin{proof}[Proof of Propositions \ref{prop:1} and \ref{prop:2}]
	Denote $y=x+{1}/{2}$. Recall, from \S\ref{sec:2}, the observation that for any $j$, the expression
	\begin{align*}
		x^j + (x+1)^j &= (y-\tfrac{1}{2})^j + (y+\tfrac{1}{2})^j 
		= 2 y^j +  \binom{j}{2}\frac{1}{2} y^{j-2}+ \binom{j}{4}\frac{1}{8} y^{j-4} + \ldots
	\end{align*}
	is a polynomial in $y$ of the same parity as $j$. Similarly $x^j-(x+1)^j$ is a polynomial in $y$ of the opposite parity as $j$.
	
	We prove both Propositions~\ref{prop:1} and \ref{prop:2}  by induction on $k$. First, we use the above identities to note that (b) and (c) hold for $k=0$ by the initial conditions $N_0(x)=1$ and $\epsilon E_0(x)=-F_0(x)=(x+1/2)^m$.
	
	For the base case $k=1$, as well as the induction step for general $k$, assuming the truth of (b) and (c) for $k-1$, the proof is the same. It goes as follows.
	
	To calculate the right side of \eqref{eq:auxiliary-recurrences-N}, we consider the first term $b_k(x)b_k(x+1)N_{k-1}(x)$. Note that
	\begin{align*}
		x^j (x+1)^k+x^k (x+1)^j&=(x(x+1))^{\min(j,k)}(x^{|j-k|}+(x+1)^{|j-k|}) \\
		&=(y^2-\tfrac14)^{\min(j,k)}(x^{|j-k|}+(x+1)^{|j-k|}),
	\end{align*}
	is an even  function of $y$ when $j$ and $k$ have same parity, and is an odd function of $y$ when $j$ and $k$ have the opposite parity.  Thus, $b_k(x)b_k(x+1)N_{k-1}(x)$ is even whenever all terms of $b_k(x)$ have the same parity, assuming (b) for $k-1$, i.e., $N_{k-1}(y)$ is an even function of $y$. 
	
	Suppose we keep all terms of same parity as the leading term in $b_k(x)$ and make the rest of the coefficients of $b_k(x)$ to be zero. Then, $b_k(x)b_k(x+1)N_{k-1}(x)$ is a non-zero even function of $y$. Its leading monomials in $y$ can only get cancelled by the even parts (in $y$) of the last two terms $a_k(b_k(x) E_{k-1}(x)+b_k(x+1) F_{k-1}(x))$ in \eqref{eq:auxiliary-recurrences-N}. 
	
	For $\epsilon=-1$, by the assumption of (c) for $k-1$, these even parts (in $y$) come from the terms in $b_k(x)$ with degree the same parity as $m$. Overall, it means that there exists a solution to $b_k(x)$ where the degree has the same parity as $m$ and all coefficients of monomials with opposite parity as $m$ are zero.  By Theorem \ref{thm:triangular-system}, it is the only solution of the triangular system of equations. Thus, we have established both (a) and (b) for $k$ for $\epsilon=-1$. 
	
	Similarly, for $\epsilon=1$, only the monomials in $b_k(x)$ with degree having the opposite parity as $m$ have non-zero coefficients. Again, by the uniqueness of solutions from Theorem~\ref{thm:triangular-system}, we can conclude (a) and (b) for $\epsilon =1$. 
	
	To establish (c) for $k$, we use \eqref{eq:auxiliary-recurrences-E} and \eqref{eq:auxiliary-recurrences-F} to get
	\begin{align*}
		x^j E_k(x)+(x+1)^j F_k(x)&=a_k((x+1)^j E_{k-1}(x)+x^j F_{k-1}(x))+\\
		&(x^j b_k(x)+(x+1)^j b_k(x+1))N_{k-1}(x).
	\end{align*}
	Using (c) for $k-1$ as well as (a) for $k$, we get the desired result. Similarly, for $(x+1)^j E_k(x)+x^j F_k(x)$ the same can be established.
	This completes the base case as well as the induction step.
\end{proof}

The formulas for the degree of $N_1(x)$ can be summarized as follows.
\begin{equation}\label{deg-N1}
\deg N_1(x)\le
\begin{cases}
2\Big\lfloor\dfrac{m-1}{2}\Big\rfloor, & \epsilon=-1,\\
2\Big\lfloor\dfrac{m-2}{2}\Big\rfloor, & \epsilon=1.
\end{cases}
\end{equation}
In all the examples listed in the Appendix, the degree of $N_1(x)$ equals the upper bound above. In fact, we have the following conjecture.
\begin{conjecture}\label{conj:1} Let $f(x)=(x+1/2)^m$ and  $k\ge 1$. Then
\begin{equation}\label{deg-Nk}
\deg N_k(x) =
\begin{cases}
2\Big\lfloor\dfrac{m-1}{2}\Big\rfloor, & \epsilon=-1,\\
2\Big\lfloor\dfrac{m-2}{2}\Big\rfloor, & \epsilon=1.
\end{cases}
\end{equation}
\end{conjecture}

\begin{thm}\label{thm:1} Let $f(x)=(x+\frac12)^m$. Then \eqref{deg-Nk}, for $k\ge 1$,  is equivalent to the following.

		The algorithm never terminates and generates:\\
		\noindent
		For $\epsilon=-1$:
\begin{enumerate}[(i)]
\item $b_k(x)=x$ for $k\geq 2$, when  $m$ is odd.
\item  $b_k(x)=x^2+ c_k$ for $k\geq 2$, when $m$ is even.
\end{enumerate} 
		For $\epsilon=1$:	
		\begin{enumerate}[(i)]
			\item $b_k(x)=x$ for $k\geq 2$, when  $m$ is even.
			\item  $b_k(x)=x^2+ c_k$ for $k\geq 2$, when $m$ is odd.
		\end{enumerate} 
\end{thm}
\begin{proof}
	
	Under the assumption of \eqref{deg-Nk} and the initial condition $N_0(x)=1$, we conclude that $N_k(x) \neq 0$ for all $k\geq 1$ and hence Theorem~\ref{thm:triangular-system} gives that the algorithm never terminates. 
	
	We also get the required values for the degrees of $b_k(x)$ as calculated in the algorithm. The properties of the parity of $b_k(x)$ given in Propositions \ref{prop:1} and \ref{prop:2} imply the remaining aspects of the required form of $b_k(x)$.
	
	The proof of the reverse implication that the statements on $b_k(x)$ for $k\geq 2$ imply \eqref{deg-Nk} is direct from the formula for the degrees of $b_k(x)$ as calculated in the algorithm.	
\end{proof}


\section{Some open problems}\label{sec:6}
We have proved that Algorithm TCF does not terminate, provided $N_k(x)\neq 0$, and established some properties of the continued fractions obtained when $f(x)=(x+1/2)^m$. Several examples are listed in Appendix~\ref{appendix}. We  examined the values of $a_k$ and $b_k(x)$ for several functions. In this concluding section, we mention some observations that we have been unable to prove, but would be interesting to confirm. 

We reproduced many classical continued fractions by means of Algorithm TCF. However, for all the new continued fractions, we were unable to provide closed forms for $a_k$ and $b_k(x)$. Further, we are unable to prove the convergence of the continued fractions so obtained. 

We have already mentioned Conjecture~\ref{conj:1} regarding the degree of $N_k(x)$. Here is another conjecture and an open question.
\begin{conjecture}
 Let $f(x)$ be as in Appendix~\ref{appendix}. For fixed \(m\) and \(\epsilon\), the sign of \(a_k\) becomes constant for $k$ large enough. The signs appear to be given by the formula:
\[
\operatorname{sgn}(a_k)
=
\begin{cases}
(-1)^{m+1}, & \epsilon=-1;\\
(-1)^m, & \epsilon=1.
\end{cases}
\]
where $k>k_0$, for some $k_0\in \mathbb N$. 
\end{conjecture}
This sign pattern is fortunate. If true, this conjecture implies that the tails of the continued fractions are either modified $S$-fractions, or  $J$-fractions---due to the stability in signs of $a_k$, and the appropriate form of the tail. Thus, the theory of these continued fractions can be invoked to prove their convergence, provided we assume a suitable growth condition on $a_k$ (see \cite{CPVWJ2008}). 

We have numerically examined what happens to the error $\eta_k(x)$ defined in \eqref{eq:error},
as $k\to \infty$. The values appear to converge to $0$, at least for $x>1$, with the convergence rate appearing to be exponential in $k$. The following question would be interesting to answer.
\begin{question}\label{Q1}
	What is the order of decay of $\eta_k(x)$	as a function of $x$, as $k \rightarrow\infty$?
\end{question} 
The interest in Question \ref{Q1} stems from the approach of Ap\'ery (see \cite{Apery1981, Rajkumar2012}) who used the continued fraction \eqref{A7} to prove the irrationality of $\zeta(3)$. The essential idea is to write $\eta_k(n)$ as 
$$\sum_{j=0}^{\infty} \frac{1}{(n+j)^3}-\frac{p_k(n)}{q_k(n)}=\zeta(3)+r_{k,n}$$ for a rational number $r_{k,n}$. We then combine the decay of the diagonal $\eta_n(n)$ (which is available thanks to an explicit polynomial recurrence relation satisfied by a variant) with the growth of the denominator of $r_{n,n}$ (which is obtained from the explicit form of $q_n(n)$ and $p_n(n)$) to derive the irrationality of $\zeta(3)$. 

It would be interesting to find estimates for $a_k, b_k(n)$, as well as an answer to Question \ref{Q1}, and estimate the denominators of the rational approximations. It remains to be seen whether these answers can help prove the irrationality of $\zeta(5), \zeta(7)$, and so on.

\begin{appendix}
\section{Results}\label{appendix}
In this appendix, we list the initial values of several continued fractions computed using the algorithm. We indicate by $\leftrightsquigarrow$ the first few terms of the continued fraction corresponding to the relevant series. 
Some of these continued fractions are equivalent to known  continued fractions, where $\leftrightsquigarrow$ can be replaced by an equality symbol, under appropriate convergence conditions. For these, we provide a reference where a complete statement can be found. 

\subsection*{The case $\epsilon=-1$}
We first list results related to $\beta(s)$ given by
$$\beta(s)=\sum_{k=0}^{\infty} \frac{(-1)^k}{ (2k+1)^{s}}= L(s,\chi_4).$$ Here $L(s,\chi_4)$ denotes the Dirichlet $L$-function for $\chi_4$, the alternating Dirichlet character modulo $4$. We take $f(x) = (2x+1)^s$, rather than $(x+1/2)^s$, to follow Ramanujan's choices; the results are equivalent. 

Ramanujan~\cite[Corollary to Entry 29, p.~149]{Berndt1989}:
\begin{equation}
\sum_{k=0}^\infty
\frac{(-1)^k}
{2x+2k+1}
\leftrightsquigarrow
\frac{ \frac{1}{4} }{ x } \fplus \frac{ \frac{1}{4} }{ x } \fplus \frac{ 1 }{ x } \fplus \frac{ \frac{9}{4} }{ x } \fplus \frac{ 4 }{ x } \fplus \frac{ \frac{25}{4} }{ x } \fplus \fdots
\end{equation}
Ap\'ery~\cite{Apery1981} (see also the even part of \cite[Corollary to Entry 31, p.~149]{Berndt1989}):
\begin{multline}\label{A2}
\sum_{k=0}^\infty
\frac{(-1)^k}
{(2x+2k+1)^2}
\leftrightsquigarrow 
\frac{ \frac{1}{8} }{ x^{2} + \frac{3}{4} } \fplus \\
\frac{ -1 }{ x^{2} + \frac{19}{4} } \fplus 
\frac{ -16 }{ x^{2} + \frac{51}{4} } \fplus 
\frac{ -81 }{ x^{2} + \frac{99}{4} } \fplus \frac{ -256 }{ x^{2} + \frac{163}{4} } \fplus \frac{ -625 }{ x^{2} + \frac{243}{4} } \fplus \fdots
\end{multline}


\begin{multline}
\sum_{k=0}^\infty
\frac{ (-1)^k}
{(2x+2k+1)^3}
\leftrightsquigarrow
\frac{ \frac{1}{16} }{ x^{3} + \frac{3}{2} x } \fplus \frac{ -\frac{39}{16} }{ x } \fplus \frac{ \frac{256}{39} }{ x } \fplus \frac{ \frac{307}{78} }{ x } \fplus \frac{ \frac{28665}{2456} }{ x } \fplus \frac{ \frac{17365169}{1805160} }{ x } \fplus \fdots
\end{multline}

\begin{multline}
\sum_{k=0}^\infty
\frac{ (-1)^k}
{(2x+2k+1)^4}
\leftrightsquigarrow 
\frac{ \frac{1}{32} }{ x^{4} + \frac{5}{2} x ^{ 2 } -  \frac{75}{16} } \fplus \frac{ 41 }{ x^{2} + \frac{2155}{164} } \fplus \frac{ -\frac{107779}{1681} }{ x^{2} + \frac{438535861}{17675756} } \fplus \\
\frac{ -\frac{2733591032000}{11616312841} }{ x^{2} + \frac{18304170233965571}{449121505850500} } \fplus 
\frac{ -\frac{650742089768117408089}{1085274054783765625} }{ x^{2} + \frac{1528469853341990045995511491}{25159641555572852248514500} } \fplus
\\
 \frac{ -\frac{46128393092294451242708316022278000}{36454361482169088481783517459881} }{ x^{2} + \frac{7771022221038384932476331145027175603461}{91682372652118658926333640611115618396} } \fplus \fdots
\end{multline}
\begin{multline}
\sum_{k=0}^\infty
\frac{ (-1)^k}
{(2x+2k+1)^5}
\leftrightsquigarrow 
\frac{ \frac{1}{64} }{ x^{5} + \frac{15}{4} x ^{ 3 } - \frac{125}{16} x } \fplus \frac{ \frac{5685}{64} }{ x } \fplus \frac{ \frac{18220}{1137} }{ x } \fplus \frac{ \frac{559503547}{103580700} }{ x } \fplus \\
\frac{ \frac{295079926382592}{12742693282925} }{ x } \fplus 
 \frac{ \frac{113596873978187805}{8862619550508503} }{ x } \fplus \frac{ \frac{68830443654276516124025}{2106861511597473933328} }{ x } \fplus \fdots
\end{multline}


\subsection*{The case $\epsilon=1$}
In the following, we take $f(x) = (x+1/2)^s$, to obtain results related to the Riemann zeta function $\zeta(s)$, given by  $$\zeta(s)=\sum_{k=1}^{\infty} \frac{1}{k^{s}}.$$

Stieltjes (see~\cite[Corollary to Entry 30, p.~150 ]{Berndt1989}):
\begin{equation}\label{A6}
\sum_{k=0}^\infty
\frac{ 1}
{(x+k+1/2)^2}
\leftrightsquigarrow
\frac{ 1 }{ x } \fplus \frac{ \frac{1}{12} }{ x } \fplus \frac{ \frac{4}{15} }{ x } \fplus \frac{ \frac{81}{140} }{ x } \fplus \frac{ \frac{64}{63} }{ x } \fplus \frac{ \frac{625}{396} }{ x } \fplus \fdots .
\end{equation}
Ramanujan~\cite[Chapter 12, Entry 32 (iii)]{Berndt1989}:
\begin{multline}\label{A7}
\sum_{k=0}^\infty
\frac{ 1}
{(x+k+1/2)^3}
\leftrightsquigarrow\\
\frac{ \frac{1}{2} }{ x^{2} + \frac{1}{4} } \fplus
 \frac{ -\frac{1}{12} }{ x^{2} + \frac{5}{4} } \fplus 
\frac{ -\frac{16}{15} }{ x^{2} + \frac{13}{4} } \fplus 
\frac{ -\frac{729}{140} }{ x^{2} + \frac{25}{4} } \fplus \frac{ -\frac{1024}{63} }{ x^{2} + \frac{41}{4} } \fplus \frac{ -\frac{15625}{396} }{ x^{2} + \frac{61}{4} } \fplus \fdots .
\end{multline}

\begin{equation}
\sum_{k=0}^\infty
\frac{ 1}
{(x+k+1/2)^4}
\leftrightsquigarrow
\frac{ \frac{1}{3} }{ x^{3} + \frac{1}{2} x } \fplus \frac{ -\frac{3}{16} }{ x } \fplus \frac{ \frac{16}{9} }{ x } \fplus \frac{ \frac{35}{36} }{ x } \fplus \frac{ \frac{108}{35} }{ x } \fplus \frac{ \frac{593}{252} }{ x } \fplus \fdots .
\end{equation}

\begin{multline}
\sum_{k=0}^\infty
\frac{ 1}
{(x+k+1/2)^5}
\leftrightsquigarrow \\
\frac{ \frac{1}{4} }{ x^{4} + \frac{5}{6} x ^{ 2 } -  \frac{47}{144} } \fplus \frac{ \frac{22}{27} }{ x^{2} + \frac{917}{264} } \fplus \frac{ -\frac{32535}{7744} }{ x^{2} + \frac{677361}{106040} } \fplus \frac{ -\frac{65916928}{4356075} }{ x^{2} + \frac{73206361}{7051660} } \fplus \\
\frac{ -\frac{245260880}{6421107} }{ x^{2} + \frac{458251529427}{29777316800} } \fplus \frac{ -\frac{33213134479391559}{414269032960000} }{ x^{2} + \frac{284862092764703087647}{13310627797265915200} } \fplus \fdots .
\end{multline}

\begin{multline}
\sum_{k=0}^\infty
\frac{ 1}
{(x+k+1/2)^6}
\leftrightsquigarrow \\
\frac{ \frac{1}{5} }{ x^{5} + \frac{5}{4} x ^{ 3 } - \frac{23}{48} x } \fplus \frac{ \frac{325}{192} }{ x } \fplus \frac{ \frac{64}{15} }{ x } \fplus \frac{ \frac{1037}{780} }{ x } \fplus \frac{ \frac{81920}{13481} }{ x } \fplus \frac{ \frac{13361733}{4247552} }{ x } \fplus \fdots .
\end{multline}

\begin{multline}
\sum_{k=0}^\infty
\frac{ 1}
{(x+k+1/2)^7}
\leftrightsquigarrow \\
\frac{ \frac{1}{6} }{ x^{6} + \frac{7}{4} x ^{ 4 } - \frac{49}{80} x ^{ 2 } + \frac{1009}{320} } \fplus \frac{ -\frac{6459}{400} }{ x^{2} + \frac{56637}{8612} } \fplus 
\frac{ -\frac{2130323536}{208593405} }{ x^{2} + \frac{11933357527937}{1146646643252} } \fplus\\
 \frac{ -\frac{20703438667338343375}{638195395504998276} }{ x^{2} + \frac{629570355063410057920479}{40970759562622710048316} } \fplus 
 \frac{ -\frac{434549112025366080557221163328}{5918026187062110525506747401} }{ x^{2} + \frac{12415223739050240618442747478634205}{581189792876002149825807849352676} } \fplus 
\\ \frac{ -\frac{90841658769672843459078437834006102000759}{642113198302358855978134558377217408980} }{ x^{2} + \frac{36157403210425931126896527413070445632888230689}{1274465582563922316285839478432019593082334212} } \fplus \fdots .
\end{multline}

\begin{multline}
\sum_{k=0}^\infty
\frac{ 1}
{(x+k+1/2)^8}
\leftrightsquigarrow
\frac{ \frac{1}{7} }{ x^{7} + \frac{7}{3} x ^{ 5 } - \frac{49}{72} x ^{ 3 } + \frac{2347}{432} x } \fplus \frac{ -\frac{3394867}{103680} }{ x } \fplus\\
 \frac{ \frac{1584128}{207849} }{ x } \fplus \frac{ \frac{293898537}{176533084} }{ x } \fplus \frac{ \frac{130808036075}{13097231511} }{ x } \fplus \frac{ \frac{3019718221415077}{776386201681500} }{ x } \fplus \fdots .
\end{multline}

\begin{multline}
\sum_{k=0}^\infty
\frac{ 1}
{(x+k+1/2)^9}
\leftrightsquigarrow
\frac{ \frac{1}{8} }{ x^{8} + 3 x ^{ 6 } - \frac{5}{8} x ^{ 4 } + \frac{989}{112} x ^{ 2 } -  \frac{110345}{1792} } \fplus \frac{ \frac{11288}{21} }{ x^{2} + \frac{518958}{49385} } \fplus \\
\frac{ -\frac{775953074671}{39022051600} }{ x^{2} + \frac{586109813696529547}{38320442592627335} } \fplus 
 \frac{ -\frac{141600717243611668189923015}{2408412696365514007032964} }{ x^{2} + \frac{991214636224055807252377630685752892}{46722400533295263812063306308165749} } \fplus \\
 \frac{ -\frac{7227287477995202710855234577808256228165492568567}{58009532067668337367491003388418021492706661776} }{ x^{2} + \frac{230446094937285736254906473320177815851207455220155493156812}{8175332889729518663270810214326837438866486349539505997705} } \fplus 
 \fdots .
\end{multline}

\begin{multline}
\sum_{k=0}^\infty
\frac{ 1}
{(x+k+1/2)^{10}}
\leftrightsquigarrow \\
\frac{ \frac{1}{9} }{ x^{9} + \frac{15}{4} x ^{ 7 } - \frac{3}{8} x ^{ 5 } + \frac{439}{32} x ^{ 3 } - \frac{27915}{256} x } \fplus \\
 \frac{ \frac{1105299}{1024} }{ x } \fplus \frac{ \frac{647168}{54873} }{ x } \fplus \frac{ \frac{813509743}{407486898} }{ x } \fplus
 \frac{ \frac{1785089840581593}{120822467030360} }{ x } \fplus \frac{ \frac{43884731541624590711171}{9527229637155602374680} }{ x } \fplus \fdots .
\end{multline}

\begin{multline}
\sum_{k=0}^\infty
\frac{ 1}
{(x+k+1/2)^{11}}
\leftrightsquigarrow\\
\frac{ \frac{1}{10} }{ x^{10} + \frac{55}{12} x ^{ 8 } + 
 \frac{11}{72} x ^{ 6 } +
  \frac{17765}{864} x ^{ 4 } - \frac{3812039}{20736} x ^{ 2 } + \frac{509089135}{248832} } \fplus \\
   \frac{ -\frac{1254421055}{46656} }{ x^{2} + \frac{45894518813}{3010610532} } \fplus \frac{ -\frac{2144153071338505920}{62942887329092521} }{ x^{2} + \frac{60303663906609724871}{2867471524914839836} } \fplus \fdots .
\end{multline}

\end{appendix}


\end{document}